\documentclass[12pt]{article}

\usepackage{amssymb}
\usepackage{amsmath}
\usepackage{amsthm}
\usepackage[alphabetic, initials]{amsrefs}

\newtheorem{theorem}{Theorem}[section]
\newtheorem{lemma}[theorem]{Lemma}

\newtheorem{proposition}[theorem]{Proposition}
\newtheorem{claim}[theorem]{Claim}

\theoremstyle{definition}

\newtheorem{corollary}[theorem]{Corollary}
\newtheorem{definition}[theorem]{Definition}

\newtheorem*{remark*}{Remark}
\newtheorem*{notation*}{Notation}

\newcommand{\eps}{\varepsilon}

\newcommand{\vphi}{\varphi}

\newcommand{\Cb}{\mathbb{C}}
\newcommand{\Eb}{\mathbb{E}}

\newcommand{\Nb}{\mathbb{N}}

\newcommand{\Rb}{\mathbb{R}}
\newcommand{\Zb}{\mathbb{Z}}

\newcommand{\Z}{\mathbb{Z}}

\begin{document}
\title{On the sharpness of a three circles theorem
for discrete harmonic functions}
\author{Gabor Lippner and Dan Mangoubi}
\date{}
\maketitle
\begin{abstract}
Any three circles theorem for discrete harmonic functions must contain
an inherent error term. In this paper we find the sharp error term in an $L^2$-three circles theorem for harmonic functions defined in~$\Zb^2$. 
The proof is highly indirect due to combinatorial obstacles
and cancellations phenomena.
We exploit Newton interpolation methods and recursive arguments.
\end{abstract}

\section{Introduction}
Let $u$ be a discrete harmonic function defined in $\Zb^d$. Let $(S_n)_{n=0}^{\infty}$ be the simple random walk starting at~$0$ and let $Q_u(n):=\Eb |u(S_n)|^2$ measure the $L^2$-growth of $u$.
We recall that $Q_u$ is absolutely monotonic.
\begin{theorem}[\cite{lip-man-2015}]
\label{thm:Q-absolutely-monotonic}
For all $k\in\Nb$ 
$$Q_u^{(k)}\geq 0\ ,$$
where $Q_u^{(k)}$ denotes the $k$th forward discrete derivative.
\end{theorem}
As a formal corollary we concluded a three circles theorem
with an inherent error term. 
\begin{corollary}[\cite{lip-man-2015}*{Proof of Theorem 1.11}]
\label{cor:3ct}
Let $u:\Zb^d\to\Cb$ be harmonic. Then
\begin{equation}
\label{ineq:3ct-error}
\forall n\in\Nb\quad Q_u(2n)\leq 2\sqrt{Q_u(n)Q_u(4n)}+2^{-\sqrt{n}}Q_u(4n)\ .
\end{equation}
\end{corollary}
The estimate~(\ref{ineq:3ct-error})
holds for all absolutely monotonic functions \mbox{$f:\Nb\to\Rb$}
in place of $Q_u$. On the one hand, if we let $\beta>1/2$, we can always find
absolutely monotonic functions which give rise to an error term larger
than $2^{-n^{\beta}}$, e.g., $f(n)=\binom{n}{k}$ where $k$ is large (see Theorem~\ref{thm:binomial-family}).
On the other hand, there exist absolutely monotonic functions,~$f$, which satisfy \mbox{$f(2n)\leq \sqrt{f(n)f(4n)}$} without any error term, as can be easily verified for $f(n)=\binom{n+k}{k}$.
These two sequences of absolutely monotonic functions represent
the two extreme behaviors with respect to the error term in~(\ref{ineq:3ct-error}). 
Thus, we are naturally led to ask
where functions of the type $Q_u$ fall in between these two extremes.
%
%

  In this paper we prove that, in fact, the error $2^{-\sqrt{n}}$ cannot be improved in Corollary~\ref{cor:3ct}
  for $d=2$ in the following strong sense.
\begin{theorem}
\label{thm:sharp-error}
For all $\eps>0$, $A>0$ and $n_0\in\Nb$ there exists a harmonic function
$u:\Zb^2\to\Cb$ and $n>n_0$ such that
\begin{equation}
\label{ineq:3ct-sharp}
 Q_u(2n) > A\sqrt{Q_u(n)Q_u(4n)} + 2^{-n^{1/2+\eps}}Q_u(4n) \ .
\end{equation}
\end{theorem}

We recall \cite{lip-man-2015}*{Theorem 1.13} that for any given $A>0, \eps>0$, by allowing the number of dimensions, $d$, to be large enough one can find a harmonic function on $u:\Zb^d\to\Rb$ such that $Q_u$ satisfies~(\ref{ineq:3ct-sharp}).
In light of this fact, the emphasis in the present article is on the sharpness
of the error term in a \emph{given} number of  dimensions, and we resolve this question in two dimensions.

More generally, we believe that the theme of understanding the way the  optimal error is affected by the properties of the underlying random walk may be of high interest as a link between analysis
and geometry. Furthermore, this question is linked  to the following  broader question: What is the ensemble of absolutely monotonic functions
 which can be realized as $Q_u$ for some harmonic function~$u$
 on $\Zb^d$? As suggested to us by Eugenia Malinnikova, it may be true also that the sharp error term
 we find is universal in the sense that different sharp discrete
  three circles theorems on $\Zb^d$ have the same error term (see \cite{guadie-malinnikova-3ct}).
   
The route we take towards the proof of Theorem~\ref{thm:sharp-error} is highly indirect due to
combinatorial obstacles, seemingly very difficult to surpass, which we would otherwise face (see Section~\ref{sec:ideas} for a more detailed explanation).
In particular, after suggesting the candidate harmonic function $u$
with corresponding large error, we  analyze $Q_u$ via its Newton expansion, and then compare $Q_u$ with a model absolutely monotonic function by proving relevant remainders estimates. 
Since the proof incorporates several independent ingredients
we first describe the global picture in Section~\ref{sec:ideas}.
%
%
%
%
%
\subsection{Ideas and outline of proof}
\label{sec:ideas}
%
%
Natural candidates for harmonic functions exhibiting  corresponding large error terms
are discrete analogs, $Z_k$, of the continuous harmonic functions
defined by $z\mapsto z^k/k!$ in $\Cb$. We have chosen to work with a discretization
process having its roots in~\cite{lovasz-discrete-analytic-survey} and adapted in~\cite{jls-duke14}.
The advantage of working with this discretization stems from the fact that it
preserves harmonic functions.
For short we denote $Q_{Z_k}$ by $Q_k$.

It has already been noted that $f_{k}(n)=\binom{n}{k}$ gives rise to an error arbitrarily close to $2^{-\sqrt{n}}$ when $k$ is large.
 One of the key observations in the proof of Theorem~\ref{thm:sharp-error}
is that there exists a one-parameter family
of such sequences of essentially absolutely monotonic functions.
This family is explicitly given by  $f_{k,\alpha}(n) := \binom{n+\alpha k}{ k}$ with $0\leq \alpha<1/2$ (see Theorem~\ref{thm:binomial-family}). One may be surprised here as
one expects that the optimal error for $f_{k, \alpha}$ 
drops continuously from~$2^{-\sqrt{n}}$ to~$0$ as~$\alpha$ goes from
$0$ to $1$.

Once we understand the behavior of the error term for the model family,~$f_{k, \alpha}$, we aim to prove 
 that $Q_k(n)$
is comparable with the $f_{k, 1/4}(n)$ in a certain range of $n$'s. This comparison should be uniform in $k$ and, very importantly,
 be valid for $n$'s much smaller than $k^2$,
as this range is the source of the  large error
in the logarithmic convexity inequality
with $f_{k, \alpha}$ (in fact, in the range $n>k^2$ all members of the family $(f_{k,\alpha})_{0\leq \alpha\leq 1}$  are  comparable. In particular, $f_{k, \alpha}$ satisfies $f_{k,\alpha}(2n)\leq 20\sqrt{f_{k, \alpha}(n)f_{k,\alpha}(4n)}$ 
 with no error term for $n>k^2$, similarly to $f_{k, 1}$.  
 See also~\cite{lip-man-2015}*{Theorem 1.12}).

As mentioned, the proof of the above comparison is 
highly indirect. The reason we are not able to approach the problem directly lies in complicated cancellation phenomena which occur in the expressions for  $|Z_k|^2$. To get a feeling for this, try to prove that 
$|(x+iy)^k|^2 = (x^2+y^2)^k$ just by expanding the parenthesis on the left hand side. This involves cancellation for alternating sums of binomial coefficients, the discrete analogues of which are 
 involved to an extent that  we are not able to keep track of anymore.
 %
 %
We now describe in some detail the main ideas and steps in the proof
of the discussed comparison.
\begin{enumerate}
\item We expand $Q_k$ as a Newton series.
$$Q_k(n)=\sum_{j=0}^k a_{k, j}\binom{n}{k-j}\ .$$
\item\label{step:recursion} We find a recursive behavior of $Q_k$:
$a_{k, j} \approx 4^{-j}Q_j(k-j)$. This statement is made precise in
Theorem~\ref{thm:Qk-recursive}.  
\item \label{step:asymptotics}
From the preceding step we can conclude (Theorem~\ref{thm:akj-main-coefficient}) that  $a_{k, j}$ is
a polynomial in $k$ of degree $j$ of which leading term is
$(k/4)^j/j!$. At this stage it looks plausible that
$Q_k$ behaves like $\binom{n+k/4}{k}$ since 
$$\binom{n+k/4}{k}=\sum_{j=0}^k \binom{k/4}{j}\binom{n}{k-j}\ .$$

However, it is far from evident at this point whether this comparison is valid for any $n<k^2$ which is the relevant range.

\item\label{step:akj-main-contribution} It turns out that in order to
have sharp upper and lower bounds on~$Q_k(n)$ in the range $n\gg k$ one needs to precisely understand $a_{k, j}$ only for $j\ll k$.
Theorem~\ref{thm:pseudo-recursion} shows
$$Q_k(n)\approx \sum_{j<k^2/n} a_{k, j}\binom{n}{k-j}\ .$$

To prove this we  need just rough upper and lower bounds on $a_{k, j}$ for all $k, j$.
The lower bound we need is $a_{k, j}\geq 0$ which
we get with no extra work (recall that we know this a priori from Theorem~\ref{thm:Q-absolutely-monotonic} or more directly from 
Step~\ref{step:recursion}).
The rough upper bound we need, Claim~\ref{claim:akj-global-bound}, follows from Step~\ref{step:recursion}.

\item\label{step:abk-remainder} We show (Corollary~\ref{cor:sum-head-estimate}) that
$$\binom{n+k/4}{k}\approx \sum_{j<k^2/n} \binom{k/4}{j}\binom{n}{k-j}$$

\item\label{step:comparison} The comparison $Q_k\approx f_{k, 1/4}$ in a relevant range (Theorem~\ref{thm:comparison}) now follows from  
steps~\ref{step:asymptotics},~\ref{step:akj-main-contribution}
and~\ref{step:abk-remainder}.
\end{enumerate}

\paragraph{A guide to the reader.}
After reading the outline of the proof in Section~\ref{sec:ideas},
we suggest to concentrate first on the proof of the lower bound in Theorem~\ref{thm:comparison}, since all the main ideas 
are contained in that part of the proof, while it does not require 
Sections~\ref{sec:truncation} and~\ref{sec:Qk-bound}.
Although Theorem~\ref{thm:Q-absolutely-monotonic} stays in the background, we made an attempt to keep this paper
independent of~\cite{lip-man-2015}. In particular,
Theorem~\ref{thm:Q-absolutely-monotonic} for the special family~$Q_k$
follows from Theorem~\ref{thm:Qk-recursive}.
The only place where we do need three simple technical lemmas from~\cite{lip-man-2015} is in the proof of the recursive formula, Theorem~\ref{thm:Qk-recursive}.
The basic notations and objects are introduced in Sections~\ref{sec:notations} and~\ref{sec:example}.
For convenience we give the list of dependencies for sections~\ref{sec:comparison}-\ref{sec:newton-remainder}:
$\ref{sec:Qk-bound}\rightarrow\ref{sec:truncation}
\rightarrow\ref{sec:comparison-conclusion}\rightarrow\ref{sec:optimality-proof}$,
 $\ref{sec:recursion}\rightarrow\ref{sec:akj-asymptotics}\rightarrow \ref{sec:comparison-conclusion}$, $\ref{sec:newton-remainder}\rightarrow\ref{sec:truncation}$, $\ref{sec:newton-remainder}\rightarrow\ref{sec:comparison-conclusion}$ and $\ref{sec:model}\rightarrow\ref{sec:optimality-proof}$.

\paragraph{Acknowledgments.}
We would like to thank Eugenia Malinnikova 
for her influence on this work. Eugenia motivated us to look for the optimal error term, after finding together with Maru Guadie  in \cite{guadie-malinnikova-3ct}
a first formulation of a three circles theorem with an error,
asking whether $2^{-\sqrt{n}}$ could be the optimal error.
An important part of this work has been conducted while both authors
were visiting the R\'enyi Institute. We wish to thank
Mikl\'os Ab\'ert and the R\'enyi Groups and Graphs research group 
for their generous support. 
The support of the BSF is gratefully acknowledged (grants nos.\ 2010214 and 2014108).
The second author gratefully acknowledges the support of the ISF (grant no.~753/14).

\section{Basic notations and definitions}
\label{sec:notations}
We fix notations and definitions we  use throughout the paper.

Let $S=\{(0, 1), (0, -1), (1, 0), (-1, 0)\}$ be the standard
symmetric generating set for $\Zb^2$.

\begin{definition} For $u:\Zb^2\to\Cb$ and  $s\in S$ let
$$\partial_s u(x): = u(x+s)- u(x)\ .$$
\end{definition}

We now define the Laplace operator acting on functions on~$\Z^2$.
\begin{definition}
Let $u:\Zb^2\to\Cb$.
$$(\Delta u) (x):= \frac{1}{4}\sum_{s\in S} (\partial_s u)(x)\ .$$
\end{definition}

\begin{definition}
A function $u:\Zb^2\to\Cb$ is called \emph{harmonic} if 
$\Delta u = 0$.
\end{definition}

Let $(S_n)_{n=1}^{\infty}$ be the simple random walk on $\Zb^2$
starting at $0$. We measure the $L^2$-growth of $u$ in terms of
the random walk.
\begin{definition}
\label{def:Q}
$$Q_u(n):=\Eb |u(S_n)|^2.$$
\end{definition}

A central role in our analysis is played also by
the following family of sequences of binomial coefficients.
\begin{definition}
\label{def:fka}
Let $0\leq\alpha\leq 1$.
Set $f_{k, \alpha}(n):=\binom{n+\alpha k}{k}$.
\end{definition}

\section{A harmonic discretization of $z\mapsto z^k/k!$}
\label{sec:example}
In this section we define a sequence of discrete harmonic functions~$Z_k$ on $\Zb^2$, being our main object of study in order to prove Theorem~\ref{thm:sharp-error}.

We discretize $z\mapsto z^k/k!$ following the recipe  appearing in~\cite{jls-duke14}.

\begin{definition}
\label{def:Fj}
$$F_k(x):=\binom{x+\frac{k-1}{2}}{k}
=\frac{1}{k!}\prod_{j=0}^{k-1}\left(x+\frac{k-1}{2}-j\right)\ .$$
\end{definition}
$F_k$ can be thought of as a discretization of $x\to x^k/k!$.
Then, we define
\begin{definition}
$$Z_k(x, y):= \sum_{l=0}^{k} {i^{l}} F_{k-l}(x)F_l(y)\ .$$
\end{definition}
Since $z^k/k!=\sum_{l=0}^{k} i^l \frac{x^{k-l}}{(k-l)!}\frac{y^l}{l!}$ we may think of $Z_k$ as discretizing $z\mapsto z^k/k!$.
The special feature of this discretization is that it preserves harmonicity \cite{jls-duke14}*{Lemma 2.1}, i.e.,~$Z_k$ is a harmonic function.

Since the discretization process does not commute with
a homothety we need the following variant of $Z_k$.
\begin{definition}
Let
\label{def:zktilde}
$$\tilde{Z}_k(x, y) := 2^k Z_k\left(\frac{x}{2}, \frac{y}{2}\right)\ ,$$
where by abuse of notations we assume $Z_k$ is defined also for
half-integers.
\end{definition}

\begin{notation*}
For convenience we let
$$Q_k := Q_{Z_k}, \quad \tilde{Q}_k:=Q_{\tilde{Z}_k}\ .$$
\end{notation*}

\section{Comparison of $Q_k$ with a model absolutely monotonic function}
\label{sec:comparison}
In Section~\ref{sec:recursion} we derive an approximate recursive formula for~$Q_k$, leading us in Section~\ref{sec:akj-asymptotics}  to the asymptotic behavior of the Newton coefficients~$a_{k, j}$. 
In Section~\ref{sec:truncation} it gives us a way to estimate $Q_k(n)$ in the range $n\gg k$ knowing the behavior of $Q_j(n)$ only for $j\ll k$. Finally, in Section~\ref{sec:comparison-conclusion} we are able to compare 
$Q_k$ with a model essentially absolutely monotonic function.

\subsection{The recursive nature of $Q_k$}
\label{sec:recursion}
We recall that Newton's interpolation (see e.g.\ \cite{courant-john-I}*{p.~473}) gives that for any function $f:\Nb\to\Rb$ 
$$f(n)=\sum_{j=0}^{n} f^{(j)}(0)\binom{n}{j}\ ,$$
where $f^{(j)}$ denotes the $j$th forward discrete derivative of $f$.

Accordingly, we define
\begin{definition}
\label{def:akj}
The Newton coefficients $a_{k, j}$ are defined by the identity
$$\forall n\in \Nb\quad Q_k(n) = \sum_{j=0}^{k} a_{k, j}\binom{n}{k-j}$$
\end{definition}
In view of the above, $a_{k, j}=Q_k^{(k-j)}(0)$.
The following theorem shows that $Q_k$ has an approximate recursive nature.
\begin{theorem}
\label{thm:Qk-recursive}
$$a_{k, j} = 4^{-j}\tilde{Q}_j(k-j).$$
In particular, $a_{k, 0}= 1$.
\end{theorem}
\begin{proof}
We recall \cite{lip-man-2015}*{Lemma 2.4} that 
\begin{equation}
\label{eqn:derivative-laplacian}
a_{k, j} = Q_k^{(k-j)}(0)=\Delta^{k-j} |Z_k|^2 (0).
\end{equation}
On the other hand, we know \cite{lip-man-2015}*{Remark 2.3}
that for any harmonic function $f:\Zb^2\to\Cb$
\begin{equation}
\label{eqn:laplace-powers-harmonic-fcn}
\Delta^l (|f|^2) (0)= 4^{-l}\sum_{s_1, \ldots s_l} |\partial_{s_1}\ldots
\partial_{s_l} f|^2(0) \ ,
\end{equation}
where all $s_i$'s are in the symmetric generating set $S$.
In order to substitute in~(\ref{eqn:laplace-powers-harmonic-fcn}) $f=Z_k$ we readily compute \cite{lip-man-2015}*{Proposition 5.2}
\begin{equation}
\label{eqn:Zk-derivatives}
 \partial_s (Z_k)(p)=Z_{k-1}\left(p+\frac{s}{2}\right)\ .
\end{equation}

Combining~(\ref{eqn:derivative-laplacian}),~(\ref{eqn:laplace-powers-harmonic-fcn}) and~(\ref{eqn:Zk-derivatives}) together we get
\begin{align*}
a_{k, j} &= 4^{-(k-j)}\sum_{s_1, \ldots s_{k-j}} |\partial_{s_1}\ldots\partial_{s_{k-j}} Z_k|^2 (0) 
\\ 
&=4^{-(k-j)}\sum_{s_1, \ldots s_{k-j}} |Z_j 
((s_1+\ldots+s_{k-j})/2)|^2 
\\ 
&=4^{-(k-j)} \sum_{s_1, \ldots s_{k-j}} |2^{-j}\tilde{Z}_j 
(s_1+\ldots+s_{k-j})|^2 =
4^{-j}\Eb |\tilde{Z}_j(S_{k-j})|^2\ .
\end{align*}
\end{proof}

\subsection{Asymptotics of $a_{k, j}$ as a function of $k$}
\label{sec:akj-asymptotics}
%
From Theorem~\ref{thm:Qk-recursive} we can obtain
the asymptotics of $a_{k, j}$ as a function of $k$.
\begin{theorem}
\label{thm:akj-main-coefficient}
$$a_{k, j} = \frac{1}{j!}\left(\frac{k}{4}\right)^j + P_{j-1}(k)\ ,$$
where $P_{j-1}$ is a polynomial of degree $j-1$ at most.
\end{theorem}

\begin{proof}
First, let us see that $a_{k, j}$ is indeed a polynomial in $k$
of degree $j$. Momentarily, let us write $a_{k, j} = a_j(k-j)$
to  emphasize the dependence on~$k$. By Theorem~\ref{thm:Qk-recursive}
$$a_j(m) = 4^{-j}\tilde{Q}_j(m) = 4^{-j}\Eb |\tilde{Z}_j|^2(S_m)\ .$$
Consequently \cite{lip-man-2015}*{Lemma 2.4},
$$a_j^{(l)}(m) = 4^{-j}\Eb(\Delta^{l}|\tilde{Z}_j|^2)(S_{m})\ .$$
Since $\Delta^l|\tilde{Z}_j|^2$ is a polynomial of degree $2j-2l$,
we get that $a_j$ is a polynomial of degree $j$  of which leading coefficient is
$\frac{1}{j!}4^{-j}\Delta^j(|\tilde{Z}_j|^2)(0, 0)$.

On the other hand, $|\tilde{Z}_j(x, y)|^2 = (x^2+y^2)^j/(j!)^2 + P_{2j-1}(x, y)$, where $P_{2j-1}$ is a polynomial of degree at most $2j-1$. It is not difficult to check that the highest degree term of $\Delta (x^2+y^2)^j$ is
$j^2(x^2+y^2)^{j-1}$ and inductively conclude the statement of
the theorem.
\end{proof}

\subsection{A truncated Newton expansion for $Q_k$}
\label{sec:truncation}
The crucial theorem~\ref{thm:pseudo-recursion} below shows that in order to have a grasp on $Q_k(n)$
in the range $n\gg k$ it is sufficient to understand the coefficients $a_{k, j}$ for 
a range of $j$'s much smaller than $k$.
To prove it we only need rough bounds on the coefficients $a_{k, j}$
for all $k, j$ which follow from Theorem~\ref{thm:Qk-recursive}.
This section is required only for the upper bound in Theorem~\ref{thm:comparison} and can be skipped on a first reading.
\begin{claim}
\label{claim:akj-global-bound}
There exist $A_1>0, B_1>0$ such that
$$\forall k\geq j\quad 0\leq a_{k, j} \leq B_1\binom{A_1 k}{j} .$$
\end{claim}
\begin{proof}
We consider the expression for $a_{k, j}$ given in Theorem~\ref{thm:Qk-recursive}. The left hand side inequality is immediate (as is seen a priori from Theorem~\ref{thm:Q-absolutely-monotonic}). For the right hand side inequality
we see that we need to bound $\tilde{Q}_j$ from above. 
To that end we examine definition~\ref{def:Q}
and successively bound the objects involved in it.
This is done in Section~\ref{sec:Qk-bound}.
We find the estimate in Proposition~\ref{prop:Qk-rough-bound}.
 If $k-j\geq j$ we conclude that
 \mbox{$a_{k, j} < 8(20 k)^j/j!$}, 
 and if $k-j<j$, we use $(20e)^{j}\leq (20ek)^j/j!$ for $k\geq j$
 to deduce that \mbox{$a_{k, j} < 8(20ek)^j/j!$}.
 Finally, recall that by a Stirling type estimate
 we can find two positive constants $c_0, c_1$ such that
 $$c_0\left(\frac{n}{e}\right)^n \sqrt{n} \leq
  n!\leq c_1\left(\frac{n}{e}\right)^n\sqrt{n}\ .$$
  The preceding estimates  immediately give 
  $$a_{k, j}\leq 8\frac{(20ek)^j}{j!}\leq 8\frac{c_1}{c_0}\binom {20 e^2 k}{j}\ .$$
\end{proof}

\begin{theorem}
\label{thm:pseudo-recursion}
There exist $A>5, B>1$ such that if $n>2k$ then
$$\sum_{j\leq 5k^2/n}a_{k, j}\binom{n}{k-j}\leq Q_k(n)\leq B\sum_{j\leq Ak^2/n} a_{k, j}\binom{n}{k-j}\ .$$
\end{theorem}
\begin{proof}
The left hand side inequality follows from the lower
bound $a_{k, j}\geq 0$ and Definition~\ref{def:akj}.
To prove the right hand side inequality,
let $A_1>0, B_1>0$ be as in Claim~\ref{claim:akj-global-bound}.
Then, we can estimate the tail in Newton's expansion of $Q_k(n)$
as follows.
\begin{equation*}
\label{eqn:tail}
\sum_{j>6e A_1 k^2/n}\!\!\!\!\!\! a_{k, j}\binom{n}{k-j} \leq 
B_1\!\!\!\!\!\sum_{j>6e A_1 k^2/n}\! \binom{A_1k}{j}\binom{n}{k-j}\stackrel{(*)}{\leq} \frac{B_1}{2}\binom{n}{k}=\frac{B_1}{2}a_{k, 0}\binom{n}{k}\ ,
\end{equation*}
where inequality~($\ast$) follows from the remainder estimate in Lemma~\ref{lem:sum-tail-estimate}.
 Combining this together with Definition~\ref{def:akj}
we conclude the desired estimate with $A=6eA_1$ and $B=1+B_1/2$. 
\end{proof}

\subsection{Comparison of $Q_k$ with $f_{k, 1/4}$}
\label{sec:comparison-conclusion}
We compare $Q_k(n)$ with $f_{k, 1/4}(n)$ (see Definition~\ref{def:fka}). It is  crucial
that this comparison be valid for a range of $n$'s much smaller than~$k^2$. 
\begin{theorem}
\label{thm:comparison}
There exists a function $\psi:\Nb\to[0,\infty)$, monotonically increasing to $\infty$, and $C>0$ such that
 for all $k\in\Nb$ and $n>k^2/\psi(k)$
$$\frac{1}{4}\binom{n+\frac{k}{4}}{k}\leq Q_k(n)\leq C \binom{n+\frac{k}{4}}{k}\ .$$
\end{theorem}

\begin{proof}
According to  Theorem~\ref{thm:akj-main-coefficient} there exists a monotonically increasing function
\mbox{$\phi:\Nb\to\Nb$} such that if $k>\phi(j)$ then 
\begin{equation}
\label{ineq:akj-control-low-j}
\frac{1}{2}\binom{k/4}{j}<a_{k, j}<2\binom{k/4}{j}\ .
\end{equation}
To prove the upper bound on $Q_k$
let $A$ and $B$ be as in Theorem~\ref{thm:pseudo-recursion}.
Let $\psi=\lceil\phi^{-1}/A\rceil$.
If $n>k^2/\psi(k)$ and $j\leq Ak^2/n$ then $\phi(j)<k$.
Also, since $A>5$ and $\phi(k)\geq k$, $\psi(k)<k/2$.
So,  the assumption $n>k^2/\psi(k)$ implies $n>2k$.
Hence, 
by Theorem~\ref{thm:pseudo-recursion} and~(\ref{ineq:akj-control-low-j}) we have
$$Q_k(n)\leq 2B \sum_{j\leq Ak^2/n} \binom{k/4}{j}\binom{n}{k-j}
\leq 2B\binom{n+k/4}{k}\ ,$$
where the last inequality follows from Lemma~\ref{lem:a+b-choose-k-expansion}.

We now prove the lower bound on $Q_k$.
If  $n>k^2/\psi(k)$ and $j\leq 5k^2/n$ then,
again, $\phi(j)<k$ and $n>2k$. Hence, by the very same Theorem~\ref{thm:pseudo-recursion} and~(\ref{ineq:akj-control-low-j})
$$Q_k(n)\geq \frac{1}{2}\sum_{j\leq 5k^2/n} \binom{k/4}{j}\binom{n}{k-j}
\geq \frac{1}{4}\binom{n+\frac{k}{4}}{k}\ ,$$
  where the last inequality is due to Corollary~\ref{cor:sum-head-estimate}.
\end{proof}

\section{A model absolutely monotonic  family}
\label{sec:model}
We study the error term in a model family of essentially absolutely monotonic functions. We recall Definition~\ref{def:fka} and check
\begin{lemma}
$f_{k, \alpha}$ is absolutely monotonic in the range $[k, \infty)$.
\end{lemma}
\begin{proof}
Observe that $f_{k, \alpha}^{(j)}(n) = f_{k-j, \alpha}(n+j\alpha)\geq 0$ for $n\geq k$.
\end{proof}

We show that for $\alpha<1/2$ we have a saturation of the error term in the corresponding  logarithmic convexity inequality.
\begin{theorem}
\label{thm:binomial-family}
Let $\eps>0$, $C>0$ and $0\leq \alpha<1/2$.
Then, for all $k\in \Nb$ large enough and $\delta>0$ small enough
the inequality
$$ f_{k, \alpha}(2n)> C\sqrt{f_{k,\alpha}(n)f_{k, \alpha}(4n)}+
 2^{-n^{1/2+\eps}}f_{k, \alpha}(4n) $$
 holds for all $n\in[k^{2-\eps}, \delta k^2]$.
\end{theorem}
%
%
%
%
%

\begin{proof}
We have for $k$ large enough and $n>k^{2-\eps}$
\begin{align*}
\frac{f_{k,\alpha}(2n)}{f_{k,\alpha}(4n)}
&=\prod_{j=0}^{k-1} \frac{2n+\alpha k-j}
{4n+\alpha k -j}
=\frac{1}{2^k}\prod_{j=0}^{k-1}\left(1-\frac{j-\alpha k}{4n-j+\alpha k}\right)
\\
&\geq \frac{1}{2^k}\prod_{\alpha k<j\leq k-1}\left(1-\frac{j-\alpha k}{4n-j+\alpha k}\right)
\geq \frac{1}{2^k}\left(1-\frac{k-\alpha k}{4n-k+\alpha k}\right)^k
\\
&\geq \frac{1}{2^k}\left(1-\frac{k}{3n}\right)^k
\geq
\frac{1}{2^k}e^{-k^2/(2n)}
\geq 2^{-2k} 
\geq 2^{-n^{1/2+\eps}}\ .
\end{align*}

The preceding inequality holds for $0\leq \alpha\leq 1$.
On the other hand, for $\alpha<1/2$ we also have
\begin{align*}
\frac{f_{k,\alpha}(2n)^2}{f_{k,\alpha}(n)f_{k,\alpha}(4n)} 
&=
\prod_{j=0}^{k-1} \frac{(2n+\alpha k -j)^2}{(n+\alpha k - j)(4n+\alpha k - j)}
\\
&=\prod_{j=0}^{k-1} \left(1+\frac{n(j-\alpha k)}{(n+\alpha k - j)(4n+\alpha k - j)}\right)
=:\prod_{j=0}^{k-1} A_{j}
\\
&=A_{\lceil \alpha k\rceil}
\prod_{j=0}^{\lceil\alpha k\rceil-1}
A_{j}\cdot A_{2\lceil \alpha k \rceil-j}
\prod_{j=2\lceil \alpha k \rceil+1}^{k-1} A_{j}
\geq \prod_{j=2\lceil\alpha k\rceil+1}^{k-1} A_{j}\ ,
\end{align*}
where the last inequality is true due to the following elementary lemma of which proof we omit.
\begin{lemma}
Let $n\in\Nb$ and for $|x|<n$ set
$F(x)=1+\frac{nx}{(n-x)(4n-x)}$. If $y+z\geq 0$, $|y|, |z|<n$
then $F(y)F(z)\geq 1$.
\end{lemma}
From here we proceed to obtain
for $n<\delta k^2$
\begin{align*}
\prod_{j=2\lceil \alpha k\rceil+1}^{k-1} A_{j}
&\geq\prod_{j=\frac{k}{2}+\lceil \alpha k \rceil}^{k-1} A_j
\geq 
\left(1+\frac{nk/2}{4n^2}\right)^{k/2-\lceil \alpha k\rceil}
\geq 
\left(1+\frac{1}{8\delta k}\right)^{k(1/2-\alpha)/2} 
\\
&\geq
e^{(1-2\alpha)/(40\delta)} > C\ .
\end{align*}
\end{proof}

\section{Optimality of the error: Proof of Theorem~\ref{thm:sharp-error}}
\label{sec:optimality-proof}
We can now combine Theorem~\ref{thm:comparison} and the study of the error term for the model function $f_{k, 1/4}$ in Section~\ref{sec:model} to conclude that the error term
in Corollary~\ref{cor:3ct} is optimal in two dimensions.
\begin{proof}[Proof of Theorem~\ref{thm:sharp-error}]
Fix $\eps_1, C_1>0$ and let $\alpha=1/4$. Let $k_0=k(\eps_1, C_1, \alpha)$ and $\delta=\delta(\eps_1, C_1, \alpha)$ be as in Theorem~\ref{thm:binomial-family}. Let $\psi$ be as in Theorem~\ref{thm:comparison}. Let $k_1\in \Nb$ be large enough 
so that $k_1^2/\psi(k_1)<\delta k_1^2$ and $k_1^{2-\eps_1}<\delta k_1^2$. Then, for $n, 2n, 4n\in[k_1^2/\psi(k_1), \delta k_1^2]\cap [k_1^{2-\eps_1}, \delta k_1^2]$
\begin{align*}  
Q_k(2n)&\stackrel{\mbox{\scriptsize Th.}\ref{thm:comparison}}{\geq} \frac{1}{4}f_{k, 1/4}(2n) \stackrel{\mbox{\scriptsize Th.\ref{thm:binomial-family}}}
{\geq} 
\frac{C_1}{4}\sqrt{f_{k, 1/4}(n)f_{k,1/4}(4n)} + \frac{1}{4}2^{-n^{1/2+\eps_1}}f_{k, 1/4}(4n)\\
&\stackrel{\mbox{\scriptsize Th.}\ref{thm:comparison}}{\geq} C_1\tilde{C}_1\sqrt{Q_k(n)Q_k(4n)}+\tilde{C}_2 2^{-n^{1/2+\eps_1}}Q_k(4n)\ .
\end{align*}
Hence, if we let $C_1=C/\tilde{C}_1$ and $\eps_1=\eps/2$, we can take $k_1$ sufficiently large such that
$\tilde{C}_2 2^{-n^{1/2+\eps_1}} > 2^{-n^{1/2+\eps}}$ for $n>k_1^2/\psi(k_1)$ and then the stated inequality holds.
\end{proof}

%

\section{A rough upper bound on $\tilde{Q}_k$}
\label{sec:Qk-bound}
The goal of this section to prove the rough upper bound on~$\tilde{Q}_k$
in Proposition~\ref{prop:Qk-rough-bound}. This section stands
independent of all other sections and is required for the proof of the upper bound in Claim~\ref{claim:akj-global-bound}.

\subsection{A bound on $F_k$}
\label{sec:Fk-bound}
\begin{lemma}
\label{lem:Fk-bound}
$$F_j(x)^2\leq \frac{x^{2j}}{j!^2}+1\ .$$
\end{lemma}
\begin{proof}
By symmetry we may assume $x\geq 0$.
Observe that for $x\geq (j-1)/2$
each factor in Definition~\ref{def:Fj} is nonnegative and obviously
bounded by $x$. Hence $F_j(x)\leq x^j/j!$ in this case.

Else, suppose $0\leq x< (j-1)/2$.
Let $m=\lfloor \frac{j+1}{2}-x \rfloor$, and denote
by $\{x\}=x-\lfloor x\rfloor$ the fractional part of $x$.
Observe that
\begin{equation*}
|F_j(x)|=\frac{1}{j!}\prod_{l=0}^{m-1}(\{x\}+l)(l+1-\{x\})
\prod_{l=m}^{j-m-1} (\{x\}+l)\ .
\end{equation*}
Hence,
$$
F_j(x)^2\leq \frac{1}{(j!)^2}\prod_{l=1}^{m-1}(l+1)^4
\prod_{l=m}^{j-m-1} (l+1)^2 =\frac{1}{\binom{j}{m}^2}\leq 1\ .
$$
\end{proof}
\subsection{Bounding $\tilde{Z_k}$}
\label{sec:Zk-bound}
\begin{proposition}
\label{prop:Zk-bound}
$$
|\tilde{Z}_k(x, y)|^2 \leq 3\cdot 20^k\sum_{l=0}^{k}\frac{x^{2l}+y^{2l}}{4^ll!^2}\ .
$$
\end{proposition}
\begin{proof}
We apply Lemma~\ref{lem:Fk-bound} to bound $Z_k(x, y)$.
\begingroup
\allowdisplaybreaks
\begin{align*}
|Z_k(x, y)|^2&\leq (k+1)\sum_{l=0}^k F_{k-l}(x)^2F_{l}(y)^2
\\
&\leq (k+1)\sum_{l=0}^{k} \left(\frac{x^{2k-2l}}{(k-l)!^2}+1\right)\left(\frac{y^{2l}}{l!^2}+1\right)
\\
&=(k+1)\sum_{l=0}^k \frac{x^{2k-2l}y^{2l}}{(k-l)!^2l!^2}
+(k+1)\sum_{l=0}^{k} \frac{x^{2l}+y^{2l}}{l!^2} + (k+1)^2
\\
\mbox{\footnotesize (by Young's ineq.)}&\leq \frac{k+1}{k}\sum_{l=0}^k \frac{(k-l)x^{2k}+ly^{2k}}{(k-l)!^2l!^2}
+ 2(k+1)^2\sum_{l=0}^{k} \frac{x^{2l}+y^{2l}}{l!^2} 
\\
&= \frac{k+1}{k}\sum_{l=0}^k \frac{l(x^{2k}+y^{2k})}{l!^2(k-l)!^2}
+2(k+1)^2\sum_{l=0}^{k} \frac{x^{2l}+y^{2l}}{l!^2}
\\
&\leq (k+1)\frac{x^{2k}+y^{2k}}{(k)!^2}\sum_{l=0}^{k} \binom{k}{l}^2 
+2(k+1)^2\sum_{l=0}^{k} \frac{x^{2l}+y^{2l}}{l!^2}
\\
&=(k+1)\binom{2k}{k}\frac{x^{2k}+y^{2k}}{(k)!^2}
+2(k+1)^2\sum_{l=0}^{k}\frac{x^{2l}+y^{2l}}{l!^2}
\\
&\leq 3\cdot 8^k \sum_{l=0}^{k}\frac{x^{2l}+y^{2l}}{l!^2} \ .
\end{align*}
\endgroup
It remains to recall Definition~\ref{def:zktilde}  to get 
the desired estimate.

\end{proof}
\subsection{Bounding $\tilde{Q}_k$}
We start with a standard random walk estimate.
\begin{lemma}
\label{lem:std-rw-estimates}
\begin{align*}
\Eb x^{2k}(S_n)\leq \frac{(2k)!}{4^k k!}n^k\ .
\end{align*}
\end{lemma}
\begin{proof}
$X_n:=x(S_n)$ is a one dimensional lazy random walk of which generating function is  $\vphi_n(t):=\Eb e^{tX_n} = \cosh^{2n}(t/2)$.
We can express its $2k$th moment as
$\Eb X_n^{2k} = \vphi_n^{(2k)}(0) \ .$
A direct calculation yields
$$\vphi_{n}^{(2k)}(0)=4^{-k}\sum_{j_1+\ldots+j_{2n}=k} \binom{2k}{2j_1, 2j_2, \ldots, 2j_{2n}}$$
The preceding sum expression  is the number of ways partitioning $2k$ distinct
objects into $2n$ distinct drawers, where each drawer
contains an even number of objects.
Since each such partition can be obtained by first pairing
the $2k$ objects and then partitioning $k$ pairs into $2n$ distinct drawers we get that the last sum
is at most 
$\frac{(2k)!}{2^k k!}(2n)^k$.
\end{proof}

Taking expectations in Proposition~\ref{prop:Zk-bound} and applying
Lemma~\ref{lem:std-rw-estimates}, we obtain
\begin{proposition}
\label{prop:Qk-rough-bound}
$$\tilde{Q}_k(n)< 8\cdot\left\{\begin{array}{lcl}
(20 n)^k/k! &,& \mbox{if } n\geq k\\
(20e)^k &,&\mbox{if } n<k
\end{array}\right.$$
\end{proposition}
\begin{proof}
We see from Proposition~\ref{prop:Zk-bound} and Lemma~\ref{lem:std-rw-estimates} that
\begin{equation}
\label{ineq:qk}
\tilde{Q}_k(n)\leq 6\cdot 20^k\sum_{l=0}^k \frac{(2l)!}{4^{2l}l!^3} n^l
\leq 6\cdot 20^k\sum_{l=0}^k \frac{n^l}{4^l l!}
\end{equation}
Observe that if $n\geq k$ then $\frac{n^l}{l!}\leq \frac{n^k}{k!}$.
Hence, in this case
$$\tilde{Q}_k(n)\leq 6\frac{(20n)^k}{k!} \sum_{l=0}^\infty 4^{-l}
=8\frac{(20n)^k}{k!}\ .$$
If $n<k$ then $n^l/l!< k^l/l!\leq k^k/k!$. So, we obtain from~(\ref{ineq:qk}) and the estimate $k!\geq (k/e)^k$ that
$$\tilde{Q}_k(n)\leq 6 \frac{(20k)^k}{k!}\leq 6(20 e)^k\ .$$
\end{proof}

\section{A Newton remainder estimate for $\binom{a+b}{k}$}
\label{sec:newton-remainder}
In this section we analyze the remainder in Newton's expansion
for $\binom{a+b}{k}$.

It is not difficult to show
\begin{lemma}[Vandermonde's identity]
\label{lem:a+b-choose-k-expansion}
$$\binom{a+b}{k} = \sum_{j=0}^k \binom{a}{k-j}\binom{b}{j}\ .$$
\end{lemma}
\begin{proof}
For nonnegative integers $a, b$
the identity follows from combinatorial interpretation of the binomial coefficients or by induction on $a$. Then, due to the
polynomial nature of the expressions involved, it is true for all 
real numbers $a, b$.
\end{proof}

The next Lemma estimates the remainder of the above expansion.
\begin{lemma}
 \label{lem:sum-tail-estimate}
 Let $a\geq 2k$. Then,
 $$\sum_{j>6ekb/a}  \binom{a}{k-j} \binom{b}{j} \leq \frac{1}{2}\binom{a}{k}\ .$$
 \end{lemma}
 \begin{proof}
 \begin{align*}
 \binom{a}{k-j}\binom{b}{j}
 &=\binom{a}{k}\frac{\binom{k}{j}\binom{b}{j}}{\binom{a-k+j}{j}} \leq \binom{a}{k} \frac{1}{j!}\frac{k^j b^j}{(a-k)^j}\\
 &\leq  \binom{a}{k} \frac{k^j b^je^j}{j^j(a-k)^j}
 \leq \binom{a}{k}  \left(\frac{2ekb}{aj}\right)^j\ .
 \end{align*}
 Hence,
 $$\sum_{j>6ekb/a} \binom{a}{k-j}\binom{b}{j} \leq \binom{a}{k} \sum_{j=1}^{\infty} 3^{-j} =\frac{1}{2}\binom{a}{k}\ .$$   
 \end{proof}

\begin{corollary}
\label{cor:sum-head-estimate}
Let $a\geq 2 k$.
Then
$$\sum_{j\leq 6ekb/a} \binom{a}{k-j}\binom{b}{j} \geq \frac{1}{2} \binom{a+b}{k}\ .$$
\end{corollary}
\begin{proof}
From Lemmas~\ref{lem:a+b-choose-k-expansion} and~\ref{lem:sum-tail-estimate} we know that
$$\sum_{j\leq 6ekb/a}  \binom{a}{k-j}\binom{b}{j}
> \binom{a+b}{k} - \frac{1}{2}\binom{a}{k}
\geq \frac{1}{2}\binom{a+b}{k}\ .$$
\end{proof}

\vfill
\noindent Gabor Lippner,\\
Department of Mathematics, Northeastern University,
 Boston, MA, USA;
 
\noindent\texttt{g.lippner@neu.edu}
\vspace{1ex}

\noindent Dan Mangoubi,\\ 
Einstein Institute of Mathematics,
Hebrew University,
Jerusalem,
Israel;

\noindent\texttt{mangoubi@math.huji.ac.il}

\begin{bibdiv}
\begin{biblist}
\bib{courant-john-I}{book}{
   author={Courant, Richard},
   author={John, Fritz},
   title={Introduction to calculus and analysis. Vol. I},
   series={Classics in Mathematics},
   note={Reprint of the 1989 edition},
   publisher={Springer-Verlag, Berlin},
   date={1999},
   pages={xxiv+661},
}

\bib{jls-duke14}{article}{
   author={Jerison, David},
   author={Levine, Lionel},
   author={Sheffield, Scott},
   title={Internal DLA and the Gaussian free field},
   journal={Duke Math. J.},
   volume={163},
   date={2014},
   number={2},
   pages={267--308},
   issn={0012-7094},
}

\bib{guadie-malinnikova-3ct}{article}{
   author={Guadie, Maru},
   author={Malinnikova, Eugenia},
   title={On three balls theorem for discrete harmonic functions},
   journal={Comput. Methods Funct. Theory},
   volume={14},
   date={2014},
   number={4},
   pages={721--734},
}

\bib{lip-man-2015}{article}{
   author={Lippner, Gabor},
   author={Mangoubi, Dan},
   title={Harmonic functions on the lattice: Absolute monotonicity and
   propagation of smallness},
   journal={Duke Math. J.},
   volume={164},
   date={2015},
   number={13},
   pages={2577--2595},
}

\bib{lovasz-discrete-analytic-survey}{article}{
   author={Lov{\'a}sz, L{\'a}szl{\'o}},
   title={Discrete analytic functions: an exposition},
   conference={
      title={Surveys in differential geometry. Vol. IX},
   },
   book={
      series={Surv. Differ. Geom., IX},
      publisher={Int. Press, Somerville, MA},
   },
   date={2004},
   pages={241--273},
}
\end{biblist}
\end{bibdiv}
\end{document}